\documentclass[12pt]{amsart}
\usepackage[all]{xy}
\usepackage{amssymb}

\newcommand{\bbc}{\mathbb{C}}

\newcommand{\bbq}{\mathbb{Q}}

\newcommand{\bbz}{\mathbb{Z}}

\newcommand{\Res}{\mathrm{Res}}
\newcommand{\ch}{\mathrm{char}\,}

\newcommand{\ab}{\mathrm{ab}}

\newtheorem{thm}{Theorem}

\newtheorem{lem}[thm]{Lemma}

\newtheorem{question}[thm]{Question}

\newtheorem{remarkk}[thm]{Remark}

\newtheorem{examplee}[thm]{Example}

\author{Bo-Hae Im and Michael Larsen}
\date{\today}
\address{Department of Mathematics, Chung-Ang University, 221, Heukseok-dong, Dongjak-gu, Seoul, 156-756, South Korea}\email{imbh@cau.ac.kr}
\address{Department of Mathematics, Indiana University, Bloomington,
Indiana 47405, USA} \email{mjlarsen@ndiana.edu}
\thanks{Bo-Hae Im was supported by the National Research Foundation of Korea Grant funded by the Korean Government(MEST) (NRF-2011-0015557).
Michael Larsen was partially supported by NSF grants DMS-0800705 and DMS-1101424.}

\begin{document}

\title[Infinite rank of elliptic curves over $\mathbf{Q}^{\ab}$]{Infinite rank of elliptic curves over $\mathbf{Q}^{\ab}$}

\begin{abstract}
If $E$ is an elliptic curve defined over a quadratic field $K$, and the $j$-invariant of $E$ is not $0$ or $1728$, then $E(\mathbf{Q}^{\ab})$ has infinite rank.  If $E$ is an elliptic curve in Legendre form, $y^2 = x(x-1)(x-\lambda)$, where $\mathbf{Q}(\lambda)$ is a cubic field, then $E(K \mathbf{Q}^{\ab})$ has infinite rank.  If $\lambda\in K$ has a minimal polynomial $P(x)$ of degree $4$ and
$v^2 = P(u)$ is an elliptic curve of positive rank over $\bbq$, we prove that $y^2 = x(x-1)(x-\lambda)$ has infinite rank over $K\bbq^{\ab}$.

\end{abstract}

\maketitle

\section{Introduction}

In \cite{FJ}, G.\ Frey and M.\ Jarden proved that every elliptic curve $E/\bbq$
has infinite rank over $\mathbf{Q}^{\ab}$ and asked whether the same is true for all abelian varieties.  For a general number field $K$ (not necessarily contained in $\bbq^{\ab}$), the question would be whether every abelian variety $A$ over $K$ is of infinite rank over $K \bbq^{\ab}$.  An affirmative answer to this question would follow from an affirmative answer to the original question, since every $\bbq^{\ab}$-point of the Weil restriction of scalars $\Res_{K/\bbq} A$ gives a $K \bbq^{\ab}$-point of $A$.
We specialize the question to dimension $1$.

\begin{question}
\label{easier}
If $E$ is an elliptic curve over a number field $K$, must $E$ have infinite rank over
$K \bbq^{\ab}$?
\end{question}

Specializing further to the case that $K$ is abelian over $\bbq$, the question can be reformulated as:

\begin{question}
\label{abelian}
Does every elliptic curve over $\bbq^{\ab}$ have infinite rank over $\bbq^{\ab}$?
\end{question}

In a recent paper \cite{K}, E.\ Kobayashi considered Question~\ref{abelian}
when $[K:\bbq]$ is odd.  In this setting, she gave an affirmative answer,
conditional on the Birch-Swinnerton-Dyer conjecture.

We give an affirmative answer to Question~\ref{easier} when
$E$ is defined over a field $K$ of degree $\le 4$ over $\bbq$
and satisfies some auxiliary condition.
In all of our results, we can replace $\bbq^{\ab}$ by $\bbq(2)$,
the compositum of all quadratic extensions of $\bbq$.  Our strategy for
finding points over $\bbq(2)$
entails looking for $\bbq$-points on the Kummer variety $\Res_{K/\bbq} E / (\pm 1)$ by
looking for curves of genus $\le 1$ on that variety.
When $K$ is a quadratic field, $\Res_{K/\bbq} E$ is an abelian surface isomorphic, over $\bbc$, to a product of two elliptic curves.
Our construction of a curve on the Kummer surface $\Res_{K/\bbq} E / (\pm 1)$
is modelled on
the construction of a rational curve on $(E_1\times E_2)/(\pm 1)$
due to J-F. Mestre \cite{M} and to M. Kuwata and L. Wang \cite{KW}.
For $[K:\bbq] = 3$, our proof depends on an analogous construction of a rational curve on  $(E_1\times E_2\times E_3)/(\pm 1)$ which is presented in \cite{Im}.
We do not know of any rational curve on $(E_1\times E_2\times E_3\times E_4)/(\pm 1)$
for generic choices of the $E_i$, but  \cite{Im} constructs a curve of genus $1$ in this variety.

\section{A Geometric Construction}

We now recall a geometric construction of a curve in
\begin{equation}
\label{kummer}
(E_1\times \cdots \times E_n)/( \pm 1 ),
\end{equation}
where $(\pm 1)$ acts diagonally on the product.

\begin{lem}\label{curve}
Let $\bar K$ be a separably closed field with $\ch(\bar K)\neq 2$ and  for an integer $n\geq 2$, let $E_1, \ldots E_n$ be pairwise non-isomorphic elliptic curves  over $\bar K$.
Then $E_1\times \cdots \times E_n)/( \pm 1 )$ contains a curve $X$ whose normalizer has genus
$$g_n:=2^{n-3}(n-4)+1.$$

In particular, $g_2=g_3=0$ and $g_4=1$.
\end{lem}

\begin{proof}
Let $E_i$ be written in Legendre form : for $i=1,2,\ldots n$,
$$E_i : y_i^2=x_i(x_i-1)(x_i-\lambda_i), ~~\lambda_i\in\bar{K}.$$
Since the $E_i$ are non-isomorphic over $\bar{K}$, the $\lambda_i$ are distinct.

Considering $E_1\times \cdots\times E_n$ as a $(\bbz/2\bbz)^n$-cover of
$$E_1/(\pm 1)\times \cdots\times E_n/(\pm 1)\cong (\mathbf{P}^1)^n,$$
we examine the inverse image in (\ref{kummer}) of $\mathbf{P}^1$ embedded diagonally in $(\mathbf{P}^1)^n$.

An affine open set of the resulting curve has coordinate ring
$$\begin{cases}
z_{12}^2&=x^2(x-1)^2(x-\lambda_1)(x-\lambda_2) \\
 &\vdots \\
z_{1n}^2&=x^2(x-1)^2(x-\lambda_1)(x-\lambda_n),
\end{cases}$$ with $z_{12}=y_1y_2$, $\ldots,$ $z_{1n}=y_1y_n$  fixed under the action of $(\pm 1)$.  A projective non-singular model is given in homogeneous coordinates by
$$C_n : \begin{cases}
u_1^2&=(v-\lambda_1t)(v-\lambda_2t),\\
&\vdots \\
u_{n-1}^2&=(v-\lambda_1t)(v-\lambda_nt).
\end{cases}$$
Then by the Riemann-Hurwitz formula, the genus $g_n$ of $C_n$ is given by
$$2g_n-2=2^{n-1}(-2)+ n 2^{n-2}.$$
If $n=2$ or $n=3$,  then $g_n=0$ and if $n=4$, then $g_n=1$. This completes the proof.
\end{proof}

It is difficult to tell when this construction produces a curve with infinitely many rational points over $\bbq$.  We do not
use Lemma~\ref{curve} directly in what follows, but it motivates the apparently \textit{ad hoc}, explicit constructions of the remainder of the paper. Each of the following sections deals with them and the quadratic case in Section~3 shows a concrete construction which motivates other cases.

\section{The Quadratic Case}

We begin with a lemma.

\begin{lem}\label{zero} Let $k$ be a non-negative integer and $Q(u,v)\in \bbq[u,v]$ a homogeneous polynomial of degree $2(2k+1)$ satisfying the functional equation $$Q(mu,v)=m^{2k+1}Q(v,u)$$ for a fixed squarefree integer $m\neq 1$.
Then $Q(u,v)$ cannot be a perfect square in $\bbc[u,v]$.
\end{lem}
\begin{proof}
Let $i$ be the largest integer such that $v^i$ divides $Q(u,v)$.  If $i$ is odd, $Q(u,v)$ cannot be a perfect square in $\bbc[u,v]$.  We therefore assume that $i=2j$.  Without loss of generality, we may assume that the $u^{4k+2-2j}v^{2j}$ coefficient is $1$.
If $q(u,v)$ is a square root of $Q(u,v)$ over $\bbc$, then the $u^{2k+1-j}v^j$-coefficient of $q(u,v)$ is $\pm 1$.
Every automorphism $\sigma$ of the complex numbers sends $q(u,v)$ to $\pm q(u,v)$.  However, $\sigma$ fixes the $u^{2k+1-j}v^j$ coefficient of $q(u,v)$, so $\sigma$ fixes $q(u,v)$, which means $q(u,v)\in \bbq[u,v]$.
From the given functional relation, $q(u,v)$ satisfies
$$q(mu,v)=\pm \sqrt{m} (m^{k}q(v,u)),$$
which gives a contradiction since $\sqrt m\not\in\bbq$.
\end{proof}

\begin{thm}\label{main}   Let $E\colon y^2=P(x):=x^3+\alpha x+\beta$ be an elliptic curve defined over a quadratic extension $K$ of $\bbq$. If  the $j$-invariant of $E$ is not $0$ or $1728$, then $E(\bbq^{ab})$ has infinite rank.
\end{thm}

\begin{proof}
Let $K=\bbq(\sqrt{m})$, where $m\in\bbz$ is a square-free integer, and $E\colon y^2=P(x):=x^3+\alpha x+\beta$ an elliptic curve defined over $K$.  By the hypothesis on the $j$-invariant, $\alpha\neq 0$ and $\beta\neq 0$.
Replacing $\alpha$ and $\beta$ by $\lambda^4\alpha$ and $\lambda^6\beta$ for
suitable $\lambda\in K$, we may assume without loss of generality that
$\alpha,\beta \notin \bbq$.

Let $\alpha=a+c\sqrt{m}$ and $\beta=b+d\sqrt{m}$ for $a, b, c, d \in\bbq, c, d\neq
0$. Then for $x_1 := -\frac{d}{c}\in\bbq$, we have $P(x_1)\in\bbq$, and
$$\left(x_1, \sqrt{P(x_1)}\right)\in E\left(K(\sqrt{P(x_1)}\right)\subseteq E\left(\bbq^{ab}\right).$$
 Now by substituting $\alpha$ by $\gamma^4\alpha$ and $\beta$ by $\gamma^6\beta$ for $\gamma\in K$ such that
 $\gamma^4\alpha, \gamma^6\beta\notin \bbq$,
 we get an isomorphism over $K$ between $E$ and the elliptic curve
 $$E_\gamma :y^2=P_\gamma(x):=x^3+\gamma^4\alpha x+\gamma^6\beta.$$

 For each such $\gamma=u+v\sqrt{m}$ for $u,v\in \bbq$, we get a point
 \begin{equation}\label{point}
 \left(\gamma^{-2}x_\gamma, \gamma^{-3}\sqrt{P_\gamma(x_\gamma)}\right)\in
E\left(K\Bigl(\sqrt{P(x_\gamma)}\Bigr)\right)\subseteq E\left(\bbq^{ab}\right),
 \end{equation}
 where $x_\gamma\in \bbq$ and $P_\gamma(x_\gamma)\in\bbq$.

\

\

Now we show that there are infinitely many quadratic fields $L$
such that $\bbq\Bigl(\sqrt{P_\gamma(x_\gamma)}\Bigr) = L$ for some $\gamma\in K$.

For $x\in\bbq$,
we expand $P_\gamma(x)$ as $R+I\sqrt m$ where $R, I\in \bbq[u,v,x]$ and we get
$$I=xT_1(u,v)+S_1(u,v) \text{ and } R=x^3+xT_2(u,v)+S_2(u,v),$$ where
 $T_i$ and $S_i$ are homogeneous polynomials in $u$ and $v$ over $\bbq$ of degree $4$ and $6$ respectively satisfying relations:
\begin{equation}\label{rel}
T_i(mu,v)=m^2T_i(v,u), ~~~~~S_i(mu,v)=m^3S_i(v,u).
\end{equation}
We solve the equation
$I=xT_1(u,v)+S_1(u,v)=0$ for $x$ and get
$$ x_\gamma= -\dfrac{S_1(u,v)}{T_1(u,v)}.$$

We then substitute this value of $x$ into the rational part $R$ of $P_\gamma(x)$, and after clearing the
denominator by multiplying by the square $(T_1(u,v))^4$, we obtain the polynomial
$$-T_1(u,v)(S_1(u,v)^3+S_1(u,v)~T_1(u,v)^2~T_2(u,v)-S_2(u,v)~T_1(u,v)^3),$$
which we denote $Q$.  Thus,
$Q$ is homogeneous of degree $22$ over $\bbq$ and from the relation (\ref{rel}), it
satisfies \begin{equation}\label{rel2}Q(mu,v)=m^{11}Q(v,u).\end{equation}

Note that by direct computation, the coefficients of the $u^{22}$-term and $u^{21}v$-term in $Q(u,v)$ are respectively,
$$A_0=c(-d^3-adc^2+bc^3), ~~~A_1=2(-6a^2dc^2-2ad^3+5abc^3+mc^4d-9cd^2b).$$
If $Q(u,v)=0$, then $A_0=A_1=0$. Since $c\neq 0$ and $d\neq 0$, we solve $A_0=0$ for $a$ and substitute
$$a=\dfrac{bc^3-d^3}{c^2d}$$ into $A_1=0$. Then we get
$$-b^2c^6-4c^3d^3b-4d^6+mc^6d^2=0,$$
whose discriminant in $b$ is $mc^{12}d^2$ which is not a square in $\bbq$. Hence $A_1\neq 0$.
This shows that $Q(u,v)$ cannot be identically zero.
By Lemma~\ref{zero}, $Q(u,v)$ cannot be a perfect square in $\bbc[u,v]$.

 Hence $y^2-Q(u,v)$ is irreducible over $\bbc$.

  Let $f(t)\in\bbq[t]$ be the polynomial of degree $22$ in the variable $t=u/v$ obtained by replacing  $Q(u,v)$ by $Q(u,v)v^{-22}$. For a finite extension $L$ of $K$, we let
  $$H(f, L):=\{t'\in \bbq: f(t')-y^2 \text{ is
irreducible over }L\}$$
the intersection of $\bbq$ with the Hilbert set of $f$ over $L$.
 By the Hilbert irreducibility theorem (\cite[Chapter~12]{JF}), such an intersection  is non-empty.

 Hence there exists $\gamma_0=u_0+v_0\sqrt{m}\in K$ such that
 $$L_0:=\bbq\Bigl(\sqrt{P_{\gamma_0}(x_{\gamma_0})}\Bigr)=\bbq\Bigl(\sqrt{Q(u_{\gamma_0},v_{\gamma_0})}\Bigr)$$
is a quadratic  field not contained in $L$.
Inductively, we get an infinite sequence of $\gamma_k=u_k+v_k\sqrt{m}$  such that  the
fields
$$L_k=\bbq\Bigl(\sqrt{P_{\gamma_k}(x_{\gamma_k})}\Bigr) = \bbq\Bigl(\sqrt{Q(u_{\gamma_k},v_{\gamma_k})}\Bigr)$$
are all linearly disjoint.

Let $V$ be the set  $$V:=\left\{\left(\gamma_k^{-2}x_{\gamma_k},\gamma_k^{-3}\sqrt{P_{\gamma_k}(x_{\gamma_k})}\right)\in E\Bigl(K\Bigl(\sqrt{P(x_{\gamma_k})}\Bigr)\Bigr)\right\}_{k=0}^\infty.$$

By \cite[Lemma]{S}, the set $\bigcup\limits_{[L:K]\leq d}
E(L)_{tor}$ is a finite set, where the union runs all over finite
extensions $L$ of $K$ whose degree over $K$ is less than or equal to $d$. Therefore, $V$ contains only finitely many torsion points.
Then   by linear disjointness of $K L_i$ over $K$, non-torsion points
$(\gamma_k^{-2}x_{\gamma_k},\gamma_k^{-3}\sqrt{P_{\gamma_k}(x_{\gamma_k})})\in V$ are linearly independent in $E(K\bbq^{ab})$. Therefore the
rank of $E(K\bbq(2))$ is infinite, therefore, the rank of $E(K\bbq^{ab})\subseteq E(\bbq^{ab})$ is infinite.
\end{proof}

\section{The Cubic Case}

\begin{thm}\label{cubic}
Let $\lambda$ denote an element of a cubic extension $K$ of $\bbq$.  Then
$E\colon y^2 = x(x-1)(x-\lambda)$ has infinite rank over $K \bbq^{\ab}$.
\end{thm}

\begin{proof}
If $\lambda\in\bbq$, then we are done, so we assume that $\bbq(\lambda)=K$.

Let
$$L(t) := t^3 - at^2 + bt - c$$
denote the minimal polynomial of $\lambda$.  Expanding, we have
$$\Bigl(\frac{b-t^2}2+ (t-a)\lambda + \lambda^2\Bigr)^2
= M(t) - L(t)\lambda,$$
where
$$M(t) := \frac{t^4-2bt^2+8ct+b^2-4ac}4.$$
Let
$$N(t) := L(t)M(t)(M(t)-L(t)).$$
Defining
\begin{align*}
x &:= \frac{M(t)}{L(t)},\\
y &:=\frac{\Bigl(\frac{b-t^2}2+ (t-a)\lambda + \lambda^2\Bigr)}{L(t)^2}\sqrt{N(t)},\\
\end{align*}
we verify by computation that $(x,y)\in K(t,\sqrt{N(t)})^2$ lies on $E$, i.e., belongs to $E(K(t,\sqrt{N(t)}))$.
Note that $\deg N = 11$, so $w^2 - N(t)$ is irreducible in $\bbc[w,t]$.
Specializing $t$ in $\bbq$, and applying Hilbert irreducibility, as before, we obtain points of $E(K L_i)$ for an infinite sequence
of quadratic fields $L_i/\bbq$.  It follows that $E$ has infinite rank over $K\bbq(2)$ and therefore over $K \bbq^{\ab}$.

\end{proof}

\section{The Quartic Case}

\begin{thm}\label{quartic}

Let $\lambda$ denote an element generating a quartic extension $K$ of $\bbq$.  Let $P(x)$ be the (monic)
minimal polynomial of $\lambda$ over $\bbq$.  If the genus $1$ curve
\begin{equation}
\label{ec}
v^2 = P(u) :=  u^4+pu^3+qu^2+ru+s
\end{equation}
is an elliptic curve of positive rank over $\bbq$, then
$E\colon y^2 = x(x-1)(x-\lambda)$ has infinite rank over $K \bbq^{\ab}$.
\end{thm}

\begin{proof}
If $(u,v)$ satisfies (\ref{ec}), then setting
\begin{multline*}
A(u,v) = (2 u^4+ p u^3 - ru- 2 s)v \\
+ \frac{8  u^6  + 8 p u^5 +  (p^2  + 4 q) u^4 - (8 s  + 2 pr) u^2 - 8 psu+  r^2 - 4 qs}4,
\end{multline*}
\begin{multline*}
B(u,v) = (4 u^3  +  3 p u^2   +  2 qu +  r)v \\
+  4  u^5 +  5 p u^4 +   (p^2   +  4 q) u^3 +  (4 r +  pq) u^2 +  (4 s+  rp)u+  ps,
\end{multline*}
and
$$C(u,v) := \frac{-2uv-2u^3-pu^2+r}2 +(v+u^2+pu+q)\lambda + (u+p)\lambda^2 + \lambda^3,$$
we have
$$C(u,v)^2 = A(u,v) - B(u,v)\lambda$$
by explicit computation.
Thus, if $(u,v)\in\bbq^2$, we have
\begin{equation}\label{pt}
\begin{split}
 P_{(u,v)}:&=\biggl(\frac{A(u,v)}{B(u,v)},C(u,v)\sqrt{\frac{A(u,v)(A(u,v)-B(u,v))}{B(u,v)^3}}\biggr) \\
&\in E\left(K \bbq\left(\sqrt{D(u,v)}\right)\right),
\end{split}
\end{equation}
where
$$D(u,v) := A(u,v)B(u,v)(A(u,v)-B(u,v))\in\bbq[u,v].$$

We embed the function field $F$ of (\ref{ec}) in the field of Laurent series $F_\infty := \bbc((t))$ by mapping $u$ to $1/t$ and $v$ to the square root
of $P(u)$ in $\bbc((t))$ with principal term $1/t^2$.  We choose the correct square root of $P(u)$ so that this defines a discrete valuation on $F$ with respect to which $A(u,v)$, $B(u,v)$ and $A(u,v)-B(u,v)$ have value
$6$, $5$, and $6$ respectively.  It follows that $F_\infty(\sqrt{D(u,v)}) = \bbc((t^{1/2}))$.  This implies that $\sqrt{D(u,v)}$ does not lie in $F$.
Therefore, $\sqrt{D(u,v)}\not\in F$.   Let $X$ denote the projective non-singular curve over $\bbc$ with function field $F[z]/(z^2-D(u,v))$.  Then there exists a
morphism from $X$ to the projective non-singular curve with function field $F$, which is ramified at $F_\infty$.  It follows that the genus of $X$ is at least $2$.
By Faltings' theorem \cite{F}, $X(\bbq(\sqrt D))$ is finite for all $D\in\bbq$.  If there are infinitely many $\bbq$-points $\{Q_k:=(u_k, v_k)\}_{k=1}^\infty$ on (\ref{ec}), their inverse images generate
infinitely many different quadratic extensions  of $\bbq$, and so the points
 $\{P_{(u_k,v_k)}\}_{k=1}^\infty$ of $E$ in (\ref{pt}) are defined over different quadratic extensions
 $K \bbq(\sqrt{D(u_k,v_k)})$ of $\bbq$. By \cite[Lemma]{S} again, it follows that $E(K \bbq(2))$ has infinite rank.
\end{proof}

\end{document}